\documentclass{amsart}

\usepackage{enumerate, amsmath, amsfonts, amssymb, amsthm, graphics, graphicx, xcolor, url, bbm, hyperref, wasysym}
\hypersetup{colorlinks=true, citecolor=darkblue, linkcolor=darkblue}
\graphicspath{{figures/}}

\newtheorem{theorem}{Theorem}[section]
\newtheorem{corollary}[theorem]{Corollary}
\newtheorem{lemma}[theorem]{Lemma}
\newtheorem{proposition}[theorem]{Proposition}
\newtheorem{observation}[theorem]{Observation}

\newtheorem{definition}[theorem]{Definition}

\newtheorem{question}[theorem]{Question}

\theoremstyle{definition}
\newtheorem{example}[theorem]{Example}
\newtheorem{remark}[theorem]{Remark}

\newcommand{\gon}[1]{\mbox{$#1$-gon}}
\newcommand{\cycle}[1]{\mbox{$#1$-cycle}}
\newcommand{\tuple}[1]{\mbox{$#1$-tuple}}
\newcommand{\clique}[1]{\mbox{$#1$-clique}}
\newcommand{\simp}[1]{\mbox{$#1$-simplex}}
\newcommand{\poly}[1]{\mbox{$#1$-poly}}
\newcommand{\connected}[1]{\mbox{$#1$-connec}ted}
\newcommand{\regular}[1]{\mbox{$#1$-regular}}
\newcommand{\domino}[1]{\mbox{$#1$-domino}}
\newcommand{\dimensional}[1]{\mbox{$#1$-dimen}\-sio\-nal}
\newcommand{\face}[1]{\mbox{$#1$-face}}
\newcommand{\R}{\ensuremath{\mathbb{R}}} 
\newcommand{\N}{\ensuremath{\mathbb{N}}} 
\newcommand{\Z}{\ensuremath{\mathbb{Z}}} 
\newcommand{\RP}{\ensuremath{\mathbb{RP}}} 
\newcommand{\cF}{\ensuremath{\mathcal{F}}} 
\newcommand{\cS}{\ensuremath{\mathcal{S}}} 
\newcommand{\ssm}{\ensuremath{\smallsetminus}} 
\newcommand{\set}[2]{\ensuremath{\left\{#1\,\middle|\,#2\right\}}} 
\DeclareMathOperator{\conv}{conv} 
\DeclareMathOperator{\gr}{gr} 
\newcommand{\eqdef}{\mbox{~\raisebox{0.2ex}{\scriptsize\ensuremath{\mathrm:}}\ensuremath{=} }}
\newcommand{\Fracfloor}[2]{\ensuremath{\left\lfloor\frac{#1}{#2}\right\rfloor}} 
\newcommand{\dotprod}[2]{\langle#1 \,|\, #2\rangle} 
\newcommand{\simplex}{\triangle} 

\definecolor{darkblue}{rgb}{0,0,0.7} 
\newcommand{\darkblue}{\color{darkblue}}
\newcommand{\defn}[1]{\emph{\darkblue #1}} 
\newcommand{\fref}[1]{Figure~\ref{#1}} 
\newcommand{\svs}{\vspace{.2cm}} 
\newcommand{\ie}{\textit{i.e.}~} 

\begin{document}

\title{Polytopality and Cartesian products of graphs}
\thanks{Julian Pfeifle was partially supported by MEC grants MTM2008-03020 and MTM2009-07242, and AGAUR grant 2009 SGR 1040. Vincent Pilaud and Francisco Santos were partially supported by MEC grant MTM2008-04699-C03-02.}

\author{Julian Pfeifle}
\address{Departament de Matem\`atica Aplicada II, Universitat Polit\`ecnica de Catalunya, Barcelona, Spain}
\email{julian.pfeifle@upc.edu}

\author{Vincent Pilaud}
\address{\'Equipe Combinatoire et Optimisation, Universit\'e Pierre et Marie Curie, Paris, France}
\email{vpilaud@math.jussieu.fr}

\author{Francisco Santos}
\address{Departamento de Matem\'aticas Estad\'istica y Computaci\'on, Universidad de Cantabria, Santander, Spain}
\email{francisco.santos@unican.es}

\date{}

\maketitle

\vspace{-.5cm}
\begin{abstract}
We study the question of polytopality of graphs: when is a given graph the graph of a polytope? We first review the known necessary conditions for a graph to be polytopal, and we provide several families of graphs which satisfy all these conditions, but which nonetheless are not graphs of polytopes.

Our main contribution concerns the polytopality of Cartesian products of non-polytopal graphs. On the one hand, we show that products of simple polytopes are the only simple polytopes whose graph is a product. On the other hand, we provide a general method to construct (non-simple) polytopal products whose factors are not polytopal.
\end{abstract}

Even though graphs are perhaps the most prominent feature of polytopes, we are still far from being able to answer several basic questions regarding them. For applications, one of the most important ones is to bound the diameter of the graph in terms of the number of variables and inequalities defining the polytope~\cite{s-cehc-10}. From a theoretical point of view, it is striking that we cannot even efficiently decide whether a given graph occurs as the graph of a polytope or not.

Steinitz' theorem from 1906 completely characterizes graphs of \dimensional{3} polytopes as the 3-connected planar graphs~\cite{s-pr-22}. For higher dimensions, the situation is much more complicated: no general characterization of graphs of polytopes is known, even in dimension~$4$. In fact, Perles observed that absolutely every graph is an induced subgraph of the graph of some \dimensional{4} polytope.

\svs
In this paper, we try to shed light upon these questions and study how polytopality behaves with respect to some common operations on graphs and polytopes. We start by reviewing in Section~\ref{subsec:necessaryconditions} some necessary conditions for a graph to be polytopal: Balinski's Theorem~\cite{b-gscps-61}, the $d$-Principal Subdivision Property~\cite{b-ncp-67} and the Separation Property~\cite{k-ppg-64}. One of our goals is to construct graphs satisfying these properties, but which nonetheless are not graphs of polytopes. We say that such graphs are non-polytopal for ``non-trivial reasons''. Moreover, since polytopes of different dimensions can have the same graph, it is also interesting to study the \defn{polytopality range} of a graph, \ie the set of possible dimensions of its realizations. For example, the polytopality range of the complete graph~$K_n$ on~$n\ge 5$ vertices is $\{4,\dots,n-1\}$. Polytopes of dimension three are also special in this respect: the graph of a $3$-polytope is never the graph of a $d$-polytope for any other~$d$.

We then focus on graphs of simple polytopes. Apart from being regular, they are special in the sense that they leave no ambiguity: the whole face lattice of a simple polytope can be (efficiently) recovered from its graph~\cite{bm-ppi-87,k-swtsp-88,f-fsppt-09}. In Section~\ref{subsec:simplepolytopes}, we construct families of non-simply-polytopal graphs for non-trivial reasons. Our main tool is the remark that every induced cycle of length $3$, $4$ or~$5$ in the graph of a simple polytope defines a \dimensional{2} face.

To close the first part of the paper, we study in Section~\ref{subsec:starclique} the behavior of polytopality with respect to the star-clique operation, which replaces a vertex of degree~$d$ by a \clique{d}. In dimension~$3$, this is the usual $\Delta Y$-operation, involved in one of the proofs of Steinitz' Theorem~\cite{z-lp-95}.

\svs
The second part of this paper is dedicated to the study of the polytopality of Cartesian products of graphs. Cartesian products of polytopal graphs are automatically polytopal, and their polytopality range has been the subject of recent research~\cite{jz-ncp-00,z-ppp-04,sz-capdp,mpp-psnp}. The main contribution of this paper concerns the polytopality of Cartesian products of non-polytopal graphs. On the one hand, we show in Section~\ref{subsec:simplypolytopalproducts} that products of simple polytopes are the only simple polytopes whose graph is a product. On the other hand, we provide in Section~\ref{subsec:polytopalproducts} a general method to construct (non-simple) polytopal products whose factors are not polytopal. To illustrate the possible behavior of polytopality under Cartesian product, we discuss various examples of products of a non-polytopal graph by a segment in Section~\ref{subsec:segment}.


\section{Polytopality of graphs}\label{sec:polytopalitygraphs}

\begin{definition}
A graph~$G$ is \defn{polytopal} if it is isomorphic to the graph of some polytope~$P$. If~$P$ is \dimensional{d}, we say that~$G$ is \poly{d}topal.
\end{definition}

In small dimension, polytopality is easy to deal with. For example, \poly{2}topal graphs are exactly cycles. The first interesting question is \poly{3}topality, which is characterized by Steinitz' ``Fundamental Theorem of convex types'':

\begin{theorem}[Steinitz~\cite{s-pr-22}]\label{theo:steinitz}
A graph~$G$ is the graph of a \poly{3}tope~$P$ if and only if $G$~is planar and \connected{3}. Moreover, the combinatorial type of~$P$ is uniquely determined by~$G$.\qed
\end{theorem}

We refer to~\cite{g-cp-03,z-lp-95} for a discussion of three approaches for proving this fundamental theorem.

The first step to realizing a graph~$G$ is to understand the possible face lattice of a polytope whose graph is~$G$. For example, it is often difficult to decide which cycles of~$G$ can define \face{2}s of a \poly{d}tope realizing~$G$. In dimension~$3$, graphs of \face{2}s are characterized by the following separation condition:

\begin{theorem}[Whitney~\cite{w-nspg-32}]\label{theo:whitney}
Let~$G$ be the graph of a \poly{3}tope~$P$. The graphs of the \face{2}s of~$P$ are precisely the induced cycles in~$G$ that do not separate~$G$.\qed
\end{theorem}

In contrast to the easy $2$- and \dimensional{3} worlds, \poly{d}topality becomes much more involved as soon as~$d\ge 4$. As an illustration, the existence of neighborly polytopes (such as the well-known cyclic polytopes) proves that all possible edges can be present in the graph of a \poly{4}tope. Starting from a neighborly polytope, and stacking vertices on undesired edges, one can even learn the following:

\begin{observation}[Perles]
Every graph is an induced subgraph of the graph of a \poly{4}tope.
\end{observation}


\subsection{Necessary conditions for polytopality}\label{subsec:necessaryconditions}

It is a long-standing question of polytope theory how to determine whether a graph is~\poly{d}to\-pal or not, without enumerating all \poly{d}topes with the same number of vertices. Here we recall some general necessary conditions and apply them to discuss polytopality of small examples.

\begin{proposition}\label{prop:necessaryconditions}
A \poly{d}topal graph~$G$ satisfies the following properties:
\begin{enumerate}
\item \defn{Balinski's Theorem}: $G$~is \connected{d}~\cite{b-gscps-61}.

\item \defn{Principal Subdivision Property} ($d$-PSP): Every vertex of~$G$ is the principal vertex of a principal subdivision of~$K_{d+1}$. Here, a \defn{subdivision} of~$K_{d+1}$ is obtained by replacing edges by paths, and a \defn{principal subdivision} of~$K_{d+1}$ is a subdivision in which all edges incident to a distinguished \defn{principal vertex} are not subdivided~\cite{b-ncp-67}.

\item \defn{Separation Property}: The maximal number of components into which~$G$ may be separated by removing~$n>d$ vertices equals~$f_{d-1}\big(C_d(n)\big)$, the maximum number of facets of a \poly{d}tope with $n$~vertices~\cite{k-ppg-64}.\qed
\end{enumerate}
\end{proposition}

\begin{remark}\label{remark:3poly}
The principal subdivision property together with Steinitz' Theorem ensure that no graph of a \poly{3}tope is \poly{d}topal for~$d\ne3$. In other words, any \poly{3}tope is the unique polytopal realization of its graph. This property is also obviously true in dimension~$0$, $1$ or $2$. In contrast, it is strongly wrong in dimension~$4$ and higher as the complete graph shows: for every $n\ge 5$ and for every $d\in\{4,\dots,n-1\}$ there are polytopes whose graph is the complete graph $K_n$.
\end{remark}

Before applying Proposition~\ref{prop:necessaryconditions} on several examples, let us insist on the fact that these necessary conditions are not sufficient (see also Examples~\ref{exm:marc+antonio} and~\ref{exm:diamond}):

\begin{example}[Non-polytopality of the complete bipartite graph~\cite{b-ncp-67}]
The complete bipartite graph~$K_{m,n}$ is not polytopal, for any two integers $m,n\ge 3$, although~$K_{n,n}$ satisfies all properties of Proposition~\ref{prop:necessaryconditions} to be \poly{4}topal as soon as~$n\ge 7$.

Indeed, assume that~$K_{n,m}$ is the graph of a \poly{d}tope~$P$. Then~$d\ge4$ because~$K_{n,m}$ is non-planar. Consider the induced subgraph~$H$ of~$K_{n,m}$ corresponding to some \face{3}~$F$ of~$P$. Because~$H$ is induced and has minimum degree at least~$3$, it contains a $K_{3,3}$ minor, so $F$~was not a \face{3} after all.
\end{example}

\begin{example}[Circulant graphs]
Let~$n$ be an integer and~$S$~denote a subset of~$\left\{1,\dots,\Fracfloor{n}{2}\right\}$. The \defn{circulant} graph~$\Gamma_n(S)$ is the graph whose vertex set is~$\Z_n$, and whose edges are the pairs of vertices with difference in~$S\cup(-S)$. Observe that the degree of~$\Gamma_n(S)$ is precisely $|S\cup(-S)|$~---~in particular, the degree is odd only if~$n$ is even and~$S$ contains~$n/2$~---~and that $\Gamma_n(S)$ is connected if and only if~$S\cup\{n\}$ is relatively prime. For example, \fref{fig:circulant} represents all connected circulant graphs on at most~$8$ vertices.

\begin{figure}[htb]
	\centerline{\includegraphics[width=\textwidth]{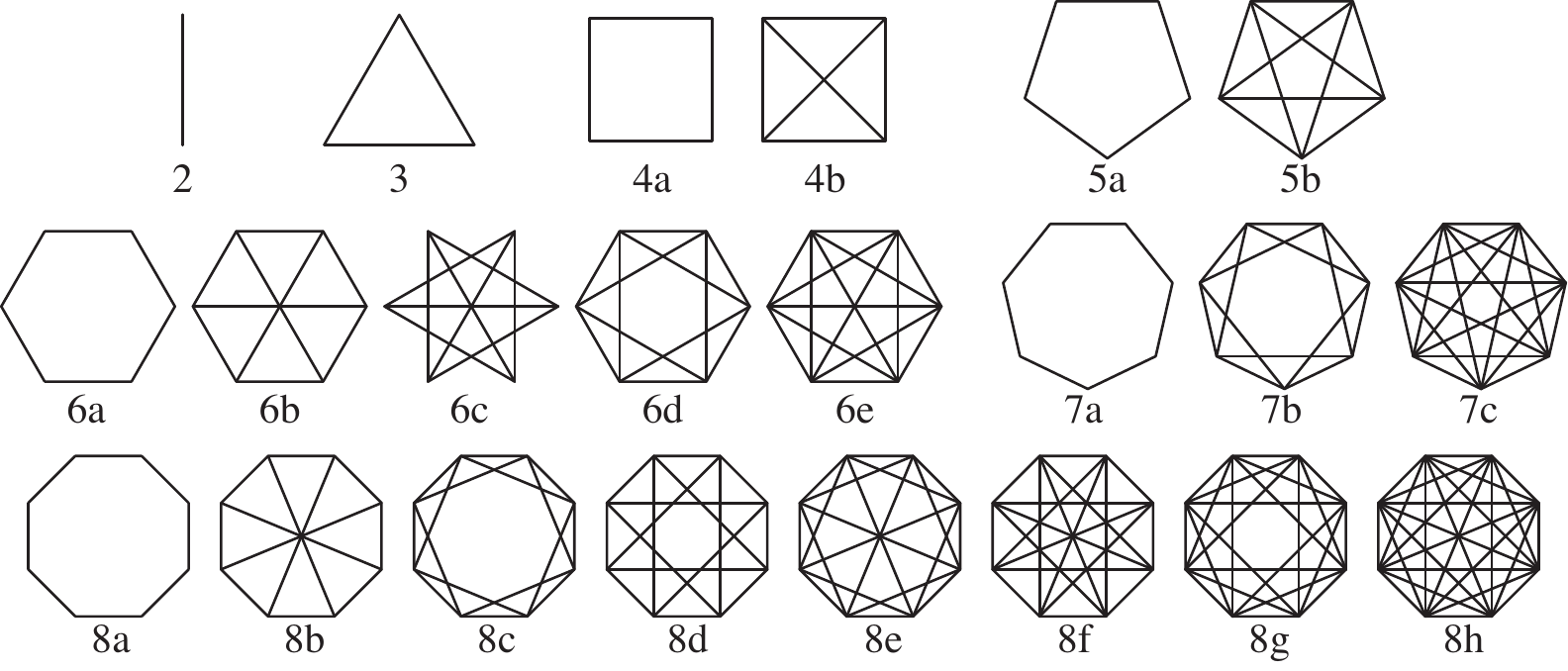}}
	\caption{Connected circulant graphs with at most~$8$ vertices.}
	\label{fig:circulant}
\end{figure}

Using Proposition~\ref{prop:necessaryconditions} we can determine the polytopality of various circulant graphs:
\begin{description}
\item[Degree~$2$] A connected circulant graph of degree~$2$ is a cycle, and thus the graph of a polygon.
\item[Degree~$3$] Up to isomorphism, the only connected circulant graphs of degree~$3$ are~$\Gamma_{2m}(1,m)$ and~$\Gamma_{4m+2}(2,2m+1)$. When~$m\ge 3$, the first one is not planar, and thus not polytopal. The second one is the graph of a prism over a \gon{(2m+1)}.
\item[Degree~$4$] As soon as we reach degree~$4$, we cannot provide a complete description of polytopal circulant graphs, but we can discuss special cases, namely the circulant graphs~$\Gamma_n(1,s)$ for~$s\in\{2,3,4\}$:
\begin{enumerate}[(a)]
\item $s=2$: For any~$m\ge 2$, the graph~$\Gamma_{2m}(1,2)$ is the graph of an antiprism over an \gon{m}. In contrast, for any~$m\ge 3$, the graph~$\Gamma_{2m+1}(1,2)$ is not polytopal: it is not planar and does not satisfy the principal subdivision property for dimension~$4$.
\item $s\in\{3,4\}$: For any~$n\ge7$, the graph~$\Gamma_n(1,3)$ is non-polytopal. Indeed, the \cycle{4}s induced by the vertices~$\{1,2,3,4\}$ and~$\{2,3,4,5\}$ should define \face{2}s of any realization (because of Theorem~\ref{theo:whitney} in dimension~$3$ and of Proposition~\ref{prop:cycles}, below, in dimension~$4$), but they intersect improperly. Similarly, for any~$n\ge9$, the graph~$\Gamma_n(1,4)$ is not polytopal.
\end{enumerate}
\item[Degree~$n-2$] The graph~$\Gamma_{2m}(1,2,\dots,m-1)$ is the only circulant graph with two vertices more than its degree. It is not planar when~$m\ge 4$ and it is not \poly{(2m-2)}topal since it does not satisfy the principal subdivision property in this dimension. However, it is always the graph of the \dimensional{m} cross-polytope, and when~$m$ is even, it is also the graph of the join of two \dimensional{(m/2)} cross-polytopes.
\item[Degree~$n-1$] The complete graph on~$n$ vertices is the graph of any neighborly polytope, and its polytopality range is~$\{4,\dots,n-1\}$ (as soon as~${n\ge 5}$).
\end{description}

\svs
The sporadic cases developed above are sufficient to determine the polytopality range of all circulant graphs on at most~$8$ vertices, except the graphs~8e and~8f of \fref{fig:circulant} that we treat separately now. None of them can be \poly{3}topal since they are not planar. We prove that they are not \poly{4}topal by discussing what could be the \face{3}s of a possible realization:
\begin{itemize}
\item We start with the graph~$\Gamma_8(1,3,4)$ represented in \fref{fig:circulant}(8f). Consider any subgraph of~$\Gamma_8(1,3,4)$ induced by~$6$ vertices. If the distance between the two missing vertices is odd (resp.~even), then the subgraph is not planar (resp.~not \connected{3}). Consequently, any subgraph of~$\Gamma_8(1,3,4)$ induced by~$7$ vertices is not planar, while any subgraph of~$\Gamma_8(1,3,4)$ induced by~$5$ vertices is not \connected{3}. Thus, the only possible \face{3}s are tetrahedra, but~$\Gamma_8(1,3,4)$ contains only~$4$ induced~$K_4$. Thus,~$\Gamma_8(1,3,4)$ is not polytopal.
\item The case of~$\Gamma_8(1,2,4)$ is more involved. Up to rotation, its only \poly{3}topal induced subgraphs are represented in \fref{fig:C8124}. Assume that the subgraph induced by~$\{0,1,2,3,4,5\}$ defines a \face{3}~$F$ in a realization~$P$ of~$\Gamma_8(1,2,4)$. Then the triangle~$123$ is a \face{2} of~$P$ and thus it should be contained in another \face{3} of~$P$. But any \face{3} which contains~$123$ also contains either~$0$ or~$4$, and thus intersects improperly with~$F$. Consequently, the subgraph induced by~$\{0,1,2,3,4,5\}$ cannot define a \face{3} of a realization of~$\Gamma_8(1,2,4)$. For the same reason, the subgraphs induced by~$\{0,1,2,3,4,6\}$ and~$\{0,1,2,3,4\}$ cannot define \face{3}s. Assume now that the subgraph induced by $\{0,1,2,3,5,6\}$ forms a \face{3} in a realization~$P$ of~$\Gamma_8(1,2,4)$. Then the triangle~$123$ is a \face{2} of~$P$ and should be contained in another \face{3} of~$P$. The only possibility is the subgraph induced by~$\{1,2,3,4,6,7\}$ which intersects improperly~$F$. Finally, the only possible \face{3}s are the two tetrahedra induced respectively by the odd and the even vertices, and thus,~$\Gamma_8(1,2,4)$ is not polytopal.
\end{itemize}
\end{example}
We summarize these results in the following proposition:

\begin{figure}
	\centerline{\includegraphics[scale=.9]{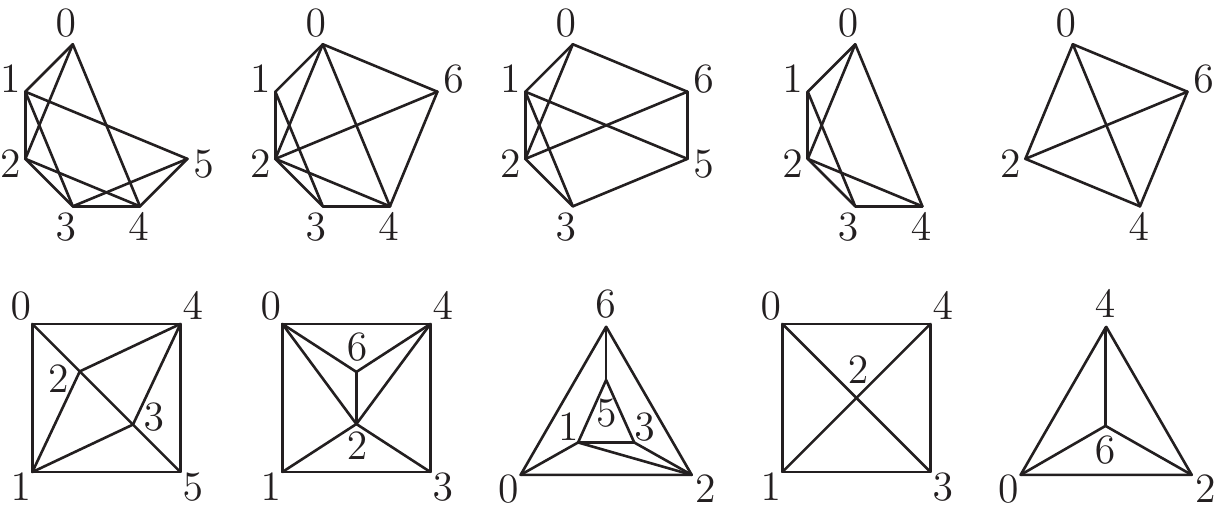}}
	\caption{The \poly{3}topal induced subgraphs of the circulant graph~$\Gamma_8(1,2,4)$. The faces of the planar drawing below each of these subgraphs are the \face{2}s of the corresponding \poly{3}tope.}
	\label{fig:C8124}
\end{figure}

\begin{proposition}
The polytopality range of all the connected circulant graphs on at most~$8$ vertices, which are depicted on \fref{fig:circulant}, is given by the following table:
\svs
\begin{center}
\begin{tabular}{|c|c|cc|cc|ccccc|}
\hline
2 & 3 & 4a & 4b & 5a & 5b & 6a & 6b & 6c & 6d & 6e \\
$\{1\}$ & $\{2\}$ & $\{2\}$ & $\{3\}$ & $\{2\}$ & $\{4\}$ & $\{2\}$ & $\emptyset$ & $\{3\}$ & $\{3\}$ & $\{4,5\}$ \\
\hline
\end{tabular}

\svs
\begin{tabular}{|ccc|cccccccc|}
\hline
7a & 7b & 7c & 8a & 8b & 8c & 8d & 8e & 8f & 8g & 8h \\
$\{2\}$ & $\emptyset$ & $\{4,5,6\}$ & $\{2\}$ & $\emptyset$ & $\{3\}$ & $\emptyset$ & $\emptyset$ & $\emptyset$ & $\{4,5\}$ & $\{4,5,6,7\}$ \\
\hline
\end{tabular}
\end{center}
\svs
\end{proposition}

\begin{example}[A graph whose polytopality range is~$\{d\}$~\cite{k-ppg-64}]\label{exm:singleton}
An interesting application of the separation property of Proposition~\ref{prop:necessaryconditions} is the possibility to construct, for any integer~$d$, a polytope whose polytopality range is exactly the singleton~$\{d\}$. The construction, proposed by Klee~\cite{k-ppg-64}, consists of stacking a vertex on all facets of the cyclic polytope~$C_d(n)$ (for example, on all facets of a simplex). The graph of the resulting polytope can be separated into~$f_{d-1}(C_d(n))$ isolated points by removing the~$n$ initial vertices, and thus is not \poly{d'}topal for~$d'<d$, by the separation property. It can not be \poly{d'}topal for $d'>d$ either, since the stacked vertices have degree~$d$ (because the cyclic polytope is simplicial). Thus, the dimension of the resulting graph is not ambiguous.
\end{example}

\begin{remark}[Polytopality range]
What subsets of $\N$ can be polytopality ranges of graphs? We know that if a polytopality range contains $1$, $2$ or $3$, then it is a singleton (Remark~\ref{remark:3poly}) and that every singleton is a polytopality range (Example~\ref{exm:singleton}), as well as any interval~$\{4,\dots,n\}$ (complete graph). We suspect that any interval~$\{m,\dots,n\}$ with $4\le m \le n$ is a polytopality range. One way of getting non-singleton polytopality ranges is to project polytopes preserving their graph. For example, \cite{mpp-psnp} obtain that for any sequence of integers $n_1,\dots,n_r$ (with $n_i\ge 2$), the product $\simplex_{n_1}\times\dots\times\simplex_{n_r}$ can be projected from dimension $\sum n_i$ until dimension $r+3$ preserving its graph (a particular example of that is the projection of the simplex until dimension~$4$). This raises the question of whether there exist graphs whose polytopality range is not an interval of $\N$.
\end{remark}


\subsection{Simple polytopes}\label{subsec:simplepolytopes}

A \poly{d}tope is \defn{simple} if its vertex figures are simplices. In other words, its facet-defining hyperplanes are in general position, so that a vertex is contained in exactly $d$~facets, and also in exactly $d$~edges (and thus the graph of a simple \poly{d}tope is \regular{d}). Surprisingly, a \regular{d} graph can be realized by at most one simple polytope:

\begin{theorem}[\cite{bm-ppi-87,k-swtsp-88}]\label{theo:kalai}
Two simple polytopes are combinatorially equivalent if and only if they have the same graph.\qed
\end{theorem}

This property, conjectured by Perles, was first proved by Blind and Mani~\cite{bm-ppi-87}. Kalai~\cite{k-swtsp-88} then gave a very simple (but exponential) algorithm for reconstructing the face lattice from the graph, and Friedman~\cite{f-fsppt-09} showed that this can even be done in polynomial time.

As mentioned previously, the first step to find a polytopal realization of a graph is often to understand what the face lattice of this realization can look like. Theorem~\ref{theo:kalai} ensures that if the realization is simple, there is only one choice. This motivates us to  temporarily restrict the study of realization of regular graphs to simple polytopes:

\begin{definition}
A graph is \defn{simply \poly{d}topal} if it is the graph of a simple \poly{d}tope.
\end{definition}

We can exploit properties of simple polytopes to obtain results on the simple polytopality of graphs. For us, the key property turns out to be that any \tuple{k} of edges incident to a vertex of a simple polytope is contained in a \face{k}. For example, this implies the following result:

\begin{proposition}\label{prop:cycles}
All induced cycles of length $3$, $4$ and $5$ in the graph of a simple \poly{d}tope~$P$ are graphs of \face{2}s of~$P$.
\end{proposition}

\begin{proof}
For \cycle{3}s, the result is immediate: any two adjacent edges of a \cycle{3} induce a \face{2}, which must be a triangle because the graph is induced.

Next, let~$\{a,b,c,d\}$ be consecutive vertices of a \cycle{4} in the graph of a simple polytope~$P$. Any pair of edges emanating from a vertex lies in a \face{2} of~$P$. Let~$C_a$ be the \face{2} of~$P$ that contains the edges~$\conv\{a,b\}$ and~$\conv\{a,d\}$. Similarly, let~$C_c$ be the \face{2} of~$P$ that contains~$\conv\{b,c\}$ and~$\conv\{c,d\}$. If~$C_a$ and~$C_c$ were distinct, they would intersect improperly, at least in the two vertices~$b$~and~$d$. Thus,~$C_a=C_c=\conv\{a,b,c,d\}$ is a \face{2} of~$P$.

The case of \cycle{5}s is a little more involved. We first show it for \poly{3}topes. If a \cycle{5}~$C$ in the graph~$G$ of a simple \poly{3}tope does not define a \face{2}, it separates~$G$ into two nonempty subgraphs~$A$~and~$B$ (Theorem~\ref{theo:whitney}). Since~$G$~is \connected{3}, both~$A$~and~$B$ are connected to~$C$ by at least three edges. But the endpoints of these six edges must be distributed among the five vertices of~$C$, so one vertex of~$C$ receives two additional edges, and this contradicts simplicity.

For the general case, we show that any \cycle{5}~$C$ in a simple polytope is contained in some \face{3}, and apply the previous argument (a face of a simple polytope is simple). First observe that any three consecutive edges in the graph of a simple polytope lie in a common \face{3}. This is true because any two adjacent edges define a \face{2}, and a \face{2} together with another adjacent edge defines a \face{3}. Thus, four of the vertices of~$C$ are already contained in a \face{3}~$F$. If the fifth vertex~$w$ of~$C$ lies outside~$F$, then the \face{2} defined by the two edges of~$C$ incident to~$w$ intersects improperly with~$F$.
\end{proof}

\begin{remark}
Observe that there is an induced \cycle{6} in the graph of the cube (resp.~an induced \cycle{p} in the graph of a double pyramid over a \cycle{p}, for~$p\ge 3$) which is not the graph of a \face{2}. It is also interesting to notice that contrarily to dimension~$3$ (Theorem~\ref{theo:whitney}), the \face{2}s of a \poly{4}tope are not characterized by a separation property: a pyramid over a cube has a non-separating induced \cycle{6} which does not define a \face{2}.
\end{remark}

\begin{corollary}\label{coro:cycles}
A simply polytopal graph cannot:
\begin{enumerate}[(i)]
\item be separated by an induced cycle of length~$3$, $4$ or~$5$.
\item contain two induced cycles of length $4$ or~$5$ which share~$3$ vertices.
\item contain an induced~$K_{2,3}$ or an induced Petersen graph. \qed
\end{enumerate}
\end{corollary}

\begin{proof}
Parts (i) and (ii) are immediate consequences of Proposition~\ref{prop:cycles} since the \face{2}s of a polytope are non-separating cycles and pairwise intersect in at most one edge. Part~(iii) arises from Part~(ii) since $K_{2,3}$ (resp. the Petersen graph) contains two induced \cycle{4}s (resp. two \cycle{5}s) which share~$3$ vertices.
\end{proof}

\begin{example}[An infinite family of non-polytopal graphs for non-trivial reasons~\cite{gn-pc-09}]\label{exm:marc+antonio}
Consider the family of graphs suggested in \fref{fig:marc+antonio}. The $n$th graph of this family is the graph~$G_n$ whose vertex set is~$\Z_{2n+3}\times\Z_2$ and where the vertex~$(x,y)$ is related with the vertices~$(x+y+1,y)$, $(x+y,y+1)$, $(x-y-1,y)$ and~$(x+y-1,y+1)$.

\begin{figure}[h]
	\centerline{\includegraphics[scale=.9]{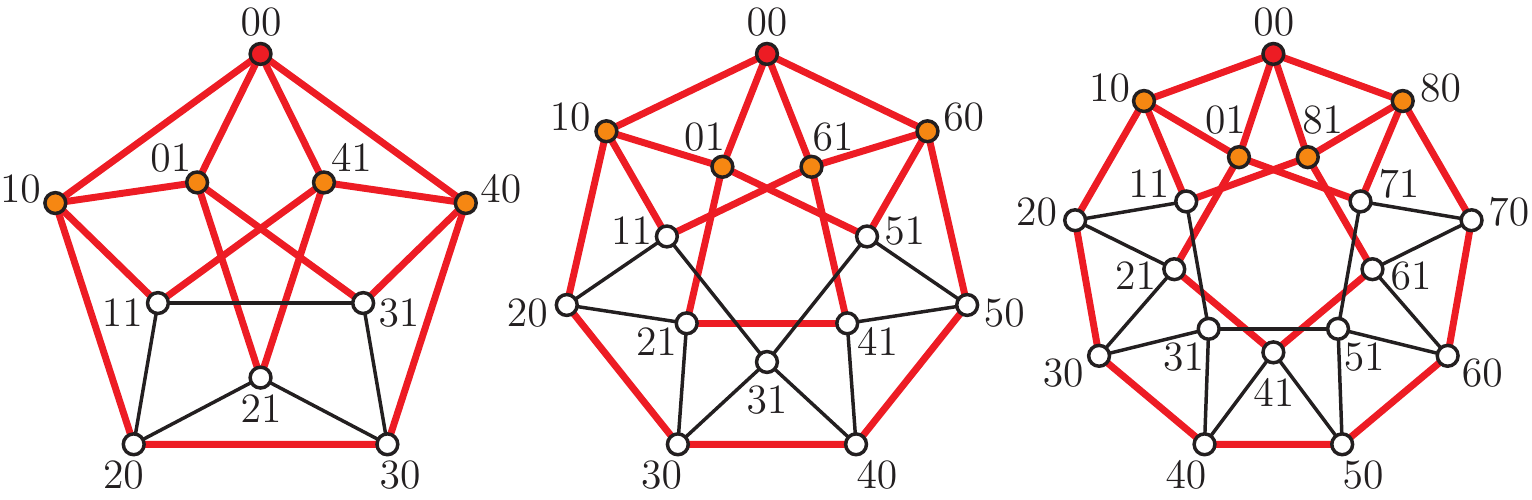}}
	\caption{An infinite family of non-polytopal graphs (for non-trivial reasons). The vertex~$00$ is the principal vertex of a principal subdivision of~$K_5$, whose edges are colored in red.}
	\label{fig:marc+antonio}
\end{figure}

Observe first that the graphs of this family satisfy all necessary conditions of Proposition~\ref{prop:necessaryconditions}:
\begin{enumerate}
\item They are \connected{4}: when we remove~$3$ vertices, either the external cycle~$\set{i0}{i\in\Z_{2n+1}}$ or the internal cycle~$\set{i1}{i\in\Z_{2n+1}}$ remains a path, to which all the vertices are connected. 
\item They satisfy the principal subdivision property for dimension~$4$: the edges of a principal subdivision of~$K_5$ with principal vertex~$00$ are colored in \fref{fig:marc+antonio}.
\item They satisfy the separation property: the cyclic \poly{4}tope on~$m$ vertices has~$\frac{m(m-3)}{2}$ facets, while removing~$m$ vertices from~$G_n$ cannot create more than~$m$ connected components.
\end{enumerate}

Consider the first graph~$G_1$ of this family (on the left in \fref{fig:marc+antonio}). Since the \cycle{5}s induced by~$\{00,10,20,30,40\}$ and~$\{00,10,20,21,41\}$ share two edges, $G_1$ is not polytopal (because of Theorem~\ref{theo:whitney} in dimension~$3$ and of Proposition~\ref{prop:cycles} in dimension~$4$). In fact, Proposition~\ref{prop:cycles} even excludes all graphs of the family:

\begin{lemma}
None of the graphs of the infinite family suggested in \fref{fig:marc+antonio} is polytopal.
\end{lemma}

\begin{proof}
Since they contain a subdivision of~$K_5$, they are not \poly{3}topal.

Denote by~$e_i$ the edge of the external cycle from vertex~$i0$ to vertex~$(i+1)0$. If the graph~$G_n$ were \poly{4}topal, then all $3$- and \cycle{4}s would define \face{2}s. Now consider two consecutive angles $e_1e_2$ and $e_2e_3$ of the external cycle. Each of them defines a \face{2} by simplicity. These two \face{2}s must in fact coincide, since $e_2$~is already contained in a square and a triangular \face{2}, none of which contain the angles~$e_1e_2$ and~$e_2e_3$. By iterating this argument, we obtain that the entire external cycle forms a \face{2}.

Consider a \face{3}~$F$ containing the external cycle. The edge $e_1$ must also be contained in either the adjacent square or the adjacent triangle; without loss of generality, let it be the square. Then the triangle adjacent to the next edge~$e_2$ must also be in~$F$ (because~$F$ already contains two of its edges). By the same reasoning, the square adjacent to~$e_3$ is also contained in~$F$, and iterating this argument (and using that $n$ is odd) shows that in fact~$F$~contains all the squares and triangles of~$G_n$, contradiction.
\end{proof}
\end{example}


\subsection{Truncation and star-clique operation}\label{subsec:starclique}

We consider the polytope~$\tau_v(P)$ obtained by cutting off a single vertex~$v$ in a polytope~$P$. The set of inequalities defining~$\tau_v(P)$ is that of~$P$ together with a new inequality satisfied strictly by all the vertices of~$P$ except~$v$. The faces of~$\tau_v(P)$ are:
\begin{enumerate}[(i)]
\item all the faces of~$P$ which do not contain~$v$;
\item the truncations~$\tau_v(F)$ of all faces~$F$ of~$P$ containing~$v$; and
\item the vertex figure of~$v$ in~$P$ together with all its faces.
\end{enumerate}
In particular, if~$v$ is a simple vertex in~$P$, then the truncation of~$v$ in~$P$ replaces~$v$ by a simplex. On the graph of~$P$, it translates into the following transformation:

\begin{definition}
Let~$G$ be a graph and~$v$ be a vertex of degree~$d$ of~$G$. The \defn{star-clique operation} (at~$v$) replaces vertex~$v$ by a \clique{d}~$K$, and assigns one edge incident to~$v$ to each vertex of~$K$. The resulting graph~$\sigma_v(G)$ has~$d-1$ more vertices and~${d \choose 2}$ more edges.
\end{definition}

\begin{remark}
In degree~$3$, star-clique operations, usually called $\Delta Y$-transforma\-tions, are used to prove Steinitz' Theorem~\ref{theo:steinitz}. The argument is that any \connected{3} and planar graph can be reduced to the complete graph~$K_4$ by a sequence of such transformations, which preserve polytopality (see~\cite{z-lp-95} for details).
\end{remark}


\begin{proposition}\label{prop:starclique}
Let~$v$ be a vertex of degree~$d$ in a graph~$G$. Then~$\sigma_v(G)$ is \poly{d}topal if and only if~$G$ is \poly{d}topal.
\end{proposition}

\begin{proof}
If a \poly{d}tope~$P$ realizes~$G$, then the truncation~$\tau_v(P)$ realizes~$\sigma_v(G)$. 

For the other direction, consider a \poly{d}tope~$Q$ which realizes~$\sigma_v(G)$. We first show that the \clique{d} replacing~$v$ forms a facet of~$Q$. Let its vertices be denoted $v_1,\dots,v_d$. Observe that all these vertices have degree $d$ in $\sigma_v(G)$. That is, $Q$ is ``simple at those vertices''. This implies that for every subset $S$ of neighbors of, say, $v_1$, there is a face of dimension $|S|$ containing $S$ and $v_1$. In particular, there is a facet $F$ of $Q$ containing $v_1,\dots,v_d$. By simplicity of all these vertices, $F$ cannot contain any other vertex.

Up to a projective transformation, we can assume that the~$d$ facets of~$Q$ adjacent to~$F$ intersect behind~$F$. Then, removing the inequality defining~$F$ from the facet description of~$Q$ creates a polytope which realizes~$G$.
\end{proof}

We can exploit Proposition~\ref{prop:starclique} to construct several families of non-polytopal graphs. We need the following lemma:

\begin{lemma}
Let~$v$ be a vertex of degree at least~$4$ in a \poly{3}topal graph~$G$. Then $\sigma_v(G)$ is not planar, and thus not \poly{3}topal.
\end{lemma}

\begin{proof}
Let~$A$ denote the cycle formed by the edges of~$G$ which are not incident to~$v$, but belong to a \face{2} incident to~$v$ in the \poly{3}tope realizing~$G$ (in other words,~$E$ denotes the link of~$v$ in the \poly{3}tope realizing~$G$). Let~$B$ denote the set of all edges of~$G$ which are neither incident to~$v$, nor contained in~$A$. Then the minor of the graph~$\sigma_v(G)$ obtained by contracting all edges of~$A$ and deleting all edges of~$B$ is the complete graph~$K_{n+1}$, where $n=\deg_G(v)$.
\end{proof}

\begin{corollary}\label{coro:starclique}
Any graph obtained from a \regular{4} \poly{3}topal graph by a finite nonempty sequence of star-clique operations is non-polytopal.
\end{corollary}

\begin{proof}
No such graph can be \poly{3}topal since it is not planar. If the resulting graph were \poly{4}topal, Proposition~\ref{prop:starclique} would assert that the original graph was also \poly{4}topal, which would contradict Remark~\ref{remark:3poly}.
\end{proof}

\begin{remark}
This corollary fails in higher dimension: the graph obtained from a \poly{d}topal graph by a star-clique operation on a vertex of degree $\delta>d$ may still be polytopal. For example, the complete graph~$K_n$ is a \regular{(n-1)} \poly{4}topal graph, and the graph~$K_{n-1} \times K_2$ obtained by a star-clique operation on a vertex of~$K_n$ is still \poly{4}topal~\cite{mpp-psnp}. 
\end{remark}

\begin{example}[Another infinite family of non-polytopal graphs for non-trivial reasons]\label{exm:diamond}
For~$n\ge3$, consider the family of graphs suggested by \fref{fig:diamond}.

They are constructed as follows: place a regular \gon{2n}~$C_{2n}$ into the plane, centered at the origin. Draw a copy~$C_{2n}'$ of~$C_{2n}$ scaled by~$\tfrac{1}{2}$ and rotated by~$\frac{\pi}{2n}$, and lift the vertices of~$C_{2n}'$ alternately to heights~$1$ and~$-1$ into the third dimension. The graph~$\davidsstar_n$ is the graph of the convex hull of the result.

In other words, the graph~$\davidsstar_n$ is the graph of the Minkowski sum of two pyramids over an \gon{n} (the first pyramid obtained as the convex hull of the even vertices of~$C_{2n}'$ together with the point~$(0,0,1)$, and the second pyramid obtained as the convex hull of the odd vertices of~$C_{2n}'$ together with the point~$(0,0,-1)$).

\begin{figure}
	\centerline{\includegraphics[width=\textwidth]{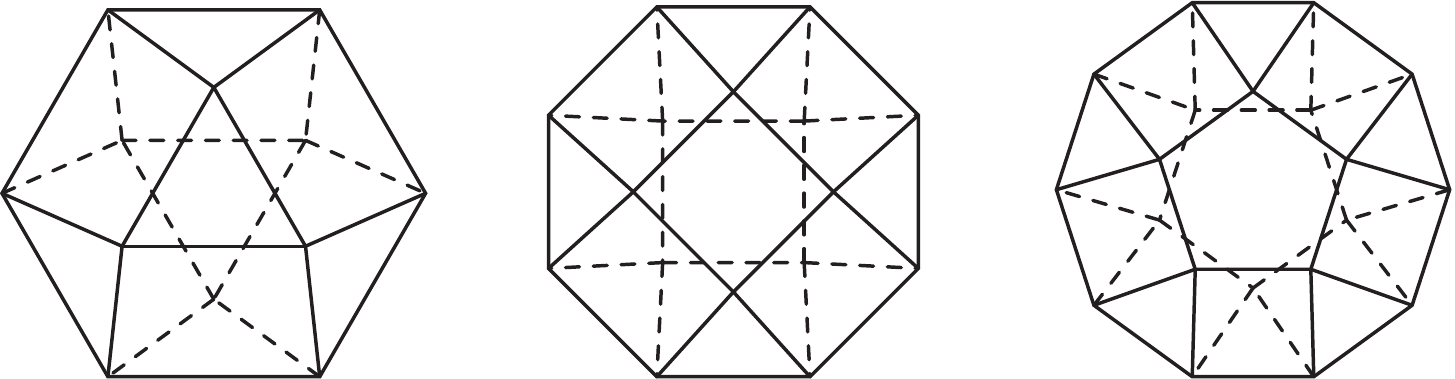}}
	\caption{The graphs $\davidsstar_n$ for $n\in\{3,4,5\}$.}
	\label{fig:diamond}
\end{figure}

Let~$\davidsstar_n^\star$~be the result of successively applying the star-clique operation to all vertices on the intermediate cycle~$C_{2n}$. Corollary~\ref{coro:starclique} ensures that~$\davidsstar_n^\star$ is not polytopal, although it satisfies all necessary conditions to be \poly{4}topal (we skip this discussion which is similar to that in Example~\ref{exm:marc+antonio}).
\end{example}


\section{Polytopality of products of graphs}\label{sec:product}

Define the \defn{Cartesian product}~$G\times H$ of two graphs~$G$~and~$H$ to be the graph whose vertex set is the product $V(G\times H) \eqdef V(G)\times V(H)$, and whose edge set is $E(G\times H) \eqdef \big(V(G)\times E(H)\big)\cup\big(E(G)\times V(H)\big)$. In other words, for~$a,c\in V(G)$ and~$b,d\in V(H)$, the vertices~$(a,b)$ and~$(c,d)$ of~$G\times H$ are adjacent if either~$a=c$ and~$\{b,d\}\in E(H)$, or~$b=d$ and~$\{a,c\}\in E(G)$. Notice that this product is usually denoted by~$G\Box H$ in graph theory. We choose to use the notation~$G\times H$ to be consistent with the Cartesian product of polytopes: if~$G$ and~$H$ are the graphs of the polytopes~$P$ and~$Q$ respectively, then the product~$G\times H$ is the graph of the product~$P\times Q$. In this section, we focus on the polytopality of products of non-polytopal graphs.

As already mentioned, the factors of a polytopal product are not necessarily polytopal: consider for example the product of a triangle by a path, or the product of a segment by two glued triangles (see \fref{fig:ctrm} and more generally Proposition~\ref{prop:subdiv}). We neutralize these elementary examples by further requiring the product~$G\times H$, or equivalently the factors~$G$ and~$H$, to be regular (the degree of a vertex~$(v,w)$ of~$G\times H$ is the sum of the degrees of the vertices~$v$ of~$G$ and~$w$ of~$H$). In this case, it is natural to investigate when such regular products can be simply polytopal. The answer is given by Theorem~\ref{theo:simplypolytopalproducts}.

\begin{figure}[htb]
	\centerline{\includegraphics[scale=.9]{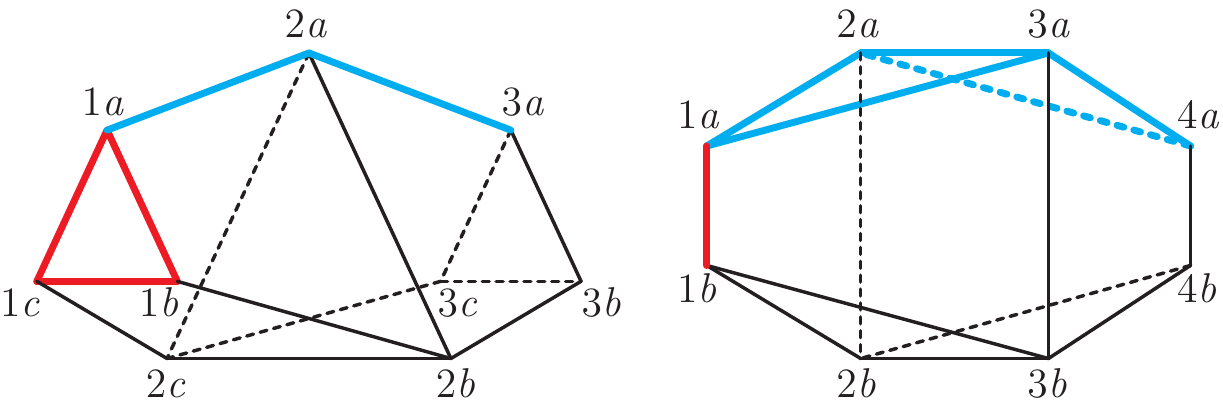}}
	\caption{Polytopal products of non-polytopal graphs: the product of a triangle~$abc$ by a path~$123$ (left) and the product of a segment~$ab$ by two glued triangles~$123$ and~$234$ (right).}
	\label{fig:ctrm}
\end{figure}

Our study of polytopality of Cartesian products of graphs was inspired by Ziegler's prototype question: 

\begin{question}[Ziegler~\cite{crm}]\label{qu:petersen}
Is the product of two Petersen graphs polytopal?
\end{question}

Incidentally, we already answered this question in the case of dimension~$6$ in Corollary~\ref{coro:cycles}: the product of two Petersen graphs cannot be simply polytopal since it contains an induced Petersen graph. However, we have no answer for dimensions~$4$ and~$5$.

Before starting, let us observe that the necessary conditions of Proposition~\ref{prop:necessaryconditions} are  preserved under Cartesian products in the following sense:

\begin{proposition}
If two graphs~$G$ and~$H$ are respectively $d$- and \connected{e}, and respectively satisfy $d$- and $e$-PSP, then their product~$G\times H$ is \connected{(d+e)} and satisfies $(d+e)$-PSP. 
\end{proposition}

\begin{proof}
The connectivity of a Cartesian product of graphs was studied in~\cite{cs-ccpg-99}. In fact, it is even proved in~\cite{s-ccpg-08} that
$$\kappa(G\times H)=\min(\kappa(G)|H|,\kappa(H)|G|,\delta(G)+\delta(H))\ge\kappa(G)+\kappa(H),$$
where~$\kappa(G)$ and~$\delta(G)$ respectively denote the connectivity and the minimum degree of a graph~$G$.

For the principal subdivision property, consider a vertex~$(v,w)$ of~$G\times H$. Choose a principal subdivision of~$K_{d+1}$ in~$G$ with principal vertex~$v$ and neighbors~$N_v$, and a principal subdivision of~$K_{e+1}$ in~$H$ with principal vertex~$w$ and neighbors~$N_w$. This gives rise to a principal subdivision of~$K_{d+e+1}$ in $G\times H$ with principal vertex~$(v,w)$ and neighbors~$(N_v\times\{w\})\cup(\{v\}\times N_w)$. Indeed, for~$x,x'\in N_v$, the vertices~$(x,w)$ and~$(x',w)$ are connected by a path in~$G\times w$ by construction; similarly, for~${y,y'\in N_w}$, the vertices~$(v,y)$ and~$(v,y')$ are connected by a path in~$v\times H$. Finally, for each~$x\in N_v$ and~$y\in N_w$, connect~$(x,w)$ to~$(v,y)$ via the path of length~$2$ that passes through~$(x,y)$. All these paths are disjoint by construction.
\end{proof}


\subsection{Simply polytopal products}\label{subsec:simplypolytopalproducts}

A product of simply polytopal graphs is automatically simply polytopal. We prove that the reciprocal statement is also true:

\begin{theorem}\label{theo:simplypolytopalproducts}
A product of graphs is simply polytopal if and only if its factors are.
\end{theorem}

Applying Theorem~\ref{theo:kalai}, we obtain a strong characterization of the simply polytopal~\mbox{products}:

\begin{corollary}\label{coro:simplypolytopalproducts}
The polytope realizing the above product of graphs is unique. Therefore, products of simple polytopes are the only simple polytopes whose graph is a product.\qed
\end{corollary}

Let~$G$ and~$H$ be two connected regular graphs of degree~$d$~and~$e$ respectively, and assume that the graph~$G\times H$~is the graph of a simple \poly{(d+e)}tope~$P$. By Proposition~\ref{prop:cycles}, for all edges~$a$ of~$G$ and~$b$ of~$H$, the \cycle{4} $a\times b$ is the graph of a \face{2} of~$P$.

\begin{observation}\label{obs:propagate}
Let $F$~be any facet of $P$, let $v$~be a vertex of~$G$, and let~$\{x,y\}$~be an edge of~$H$ such that $(v,x)\in F$ and $(v,y)\notin F$. Then~$G\times\{x\}\subset F$ and~$G\times\{y\}\cap F=\emptyset$.
\end{observation}

\begin{proof}
Since the polytope is simple, all neighbors of~$(v,x)$ except~$(v,y)$ are connected to~$(v,x)$ by an edge of~$F$. Let~$v'$ be a neighbor of~$v$ in~$G$, and let~$C$ be the \face{2}~$\conv\{v,v'\}\times\conv\{x,y\}$ of~$P$. If~$(v',y)$ were a vertex of~$F$, the intersection~$C\cap F$ would consist of exactly three vertices (because~$(v,y)\notin F$), a contradiction. In summary,~$(v',x)\in F$ and~${(v',y)\notin F}$, for all neighbors~$v'$ of~$v$. Repeating this argument and using the fact that~$G$~is connected yields~$G\times\{x\}\subset F$ and~$G\times\{y\}\cap F=\emptyset$.   
\end{proof}

\begin{lemma}\label{lem:facets}
The graph of any facet of~$P$ is either of the form~$G'\times H$ for a \regular{(d-1)} induced subgraph~$G'$ of~$G$, or of the form~$G\times H'$ for an \regular{(e-1)} induced subgraph~$H'$~of~$H$.
\end{lemma}

\begin{proof}
Assume that the graph of a facet~$F$ is not of the form~$G'\times H$.  Then there exists a vertex~$v$ of~$G$ and an edge~$\{x,y\}$ of~$H$ such that~$(v,x)\in F$ and~$(v,y)\notin F$. By Observation~\ref{obs:propagate}, the subgraph~$H'$ of~$H$ induced by the vertices~$y\in H$ such that~$G\times \{y\}\subset F$ is nonempty. We now prove that the graph~$\gr(F)$ of~$F$ is exactly~$G\times H'$. 

The inclusion~$G\times H'\subset\gr(F)$ is clear: by definition,~$G\times\{y\}$ is a subgraph of~$\gr(F)$ for any vertex~$y\in H'$. For any edge~$\{x,y\}$ of~$H'$ and any vertex~$v\in G$, the two vertices~$(v,x)$ and~$(v,y)$ are contained in~$F$, so the edge between them is an edge of~$F$; if not, we would have an improper intersection between~$F$ and this edge. 
For the other inclusion, define~$H'' \eqdef \set{y\in H}{G\times\{y\}\cap F = \emptyset}$ and let $H''' \eqdef H \ssm(H'\cup H'')$. If $H'''\ne\emptyset$, the fact that~$H$~is connected ensures that there is an edge between some vertex of~$H'''$ and either a vertex of~$H'$ or~$H''$. This contradicts Observation~\ref{obs:propagate}.

We have proved that~$G\times H'=\gr(F)$. The fact that~$F$ is a simple \poly{(d+e-1)}\-tope and the \regular{d}ity of~$G$ together ensure that~$H'$~is \regular{(e-1)}.
\end{proof}

\begin{proof}[Proof of Theorem~\ref{theo:simplypolytopalproducts}]
One direction is clear. For the other direction, proceed by induction on~$d+e$, the cases~$d=0$ and~$e=0$ being trivial.  Now assume that~$d,e\ge1$, that~$G\times H=\gr(P)$, and that $G$~is not the graph of a \poly{d}tope. By Lemma~\ref{lem:facets}, all facets of~$P$ are of the form~$G'\times H$ or~$G\times H'$, where~$G'$ (resp.~$H'$) is an induced \regular{(d-1)} (resp.~\regular{(e-1)}) subgraph of~$G$ (resp.~$H$). By induction, the second case does not arise. We fix a vertex~$w$ of~$H$. Then induction tell us that~$F_w \eqdef G'\times\{w\}$ is a face of~$P$, and~$G'\times H$ is the only facet of~$P$ that contains~$F_w$ by Lemma~\ref{lem:facets}. This cannot occur unless~$F_w$~is a facet, but this only happens in the base case~$H=\{w\}$.
\end{proof}

\begin{example}
Consider a graph~$G$ that is \regular{d}, \connected{d}, and satisfies $d$-PSP, but is not simply \poly{d}topal. Then, any product of~$G$ by a simply \poly{e}topal graph is \regular{(d+e)}, \connected{(d+e)}, satisfies $(d+e)$-PSP, but is not simply \poly{(d+e)}topal.

For example, the product of the circulant graph~$C_8(1,4)$ by the graph of the \dimensional{d} cube is a non simply polytopal graph for non-trivial reasons. For any~$m\ge 4$, the product of the circulant graph~$C_{2m}(1,m)$ by a segment is non-polytopal for non-trivial reasons.
\end{example}


\subsection{Polytopal products of non-polytopal graphs}\label{subsec:polytopalproducts}

In this section, we give a general construction to obtain polytopal products starting from a polytopal graph~$G$ and a non-polytopal one~$H$. We need the graph~$H$ to be the graph of a \defn{regular subdivision} of a polytope~$Q$, that is, the graph of the upper\footnote{The unusual convention we adopt here of defining a subdivision as the projection of the upper facets of the lifting simplifies the presentation of the construction.} envelope (the set of all upper facets with respect to the last coordinate) of the convex hull of the point set~$\set{(q,\omega(q))}{q\in V(Q)}\subset\R^{e+1}$ obtained by lifting the vertices of~$Q\subset\R^e$ according to a \defn{lifting function}~$\omega:V(Q)\to\R$.

\begin{proposition}\label{prop:subdiv}
If~$G$ is the graph of a \poly{d}tope~$P$, and~$H$ is the graph of a regular subdivision of an \poly{e}tope~$Q$, then~$G\times H$~is \poly{(d+e)}topal. In the case~$d>1$, the regular subdivision of~$Q$ can even have internal vertices.
\end{proposition}

\begin{proof}
Let $\omega:V(Q)\to\R_{>0}$ be a lifting function that induces a regular subdivision of~$Q$ with graph~$H$. Assume without loss of generality that the origin of~$\R^d$ lies in the interior of~$P$. For each~$p\in V(P)$ and~$q\in V(Q)$, we define the point~$\rho(p,q) \eqdef (\omega(q)p,q)\in\R^{d+e}$. Consider
$$R \eqdef \conv\set{\rho(p,q)}{p\in V(P), q\in V(Q)}.$$
Let~$g$ be a facet of~$Q$ defined by the linear inequality~$\dotprod{\psi}{y}\le 1$. Then the inequality~$\dotprod{(0,\psi)}{(x,y)}\le 1$ defines a facet of~$R$, with vertex set~$\set{\rho(p,q)}{p\in P, q\in g}$, and isomorphic to~$P\times g$.

Let~$f$ be a facet of~$P$ defined by the linear inequality~$\dotprod{\phi}{x}\le 1$. Let~$c$ be a cell of the subdivision of~$Q$, and let~$\psi_0 h + \dotprod{\psi}{y}\le 1$ be the linear inequality that defines the upper facet corresponding to~$c$ in the lifting. Then we claim that the linear inequality
$$\chi(x,y) = \psi_0\dotprod{\phi}{x}+\dotprod{\psi}{y} \le 1$$
selects a facet of~$R$ with vertex set~$\set{\rho(p,q)}{p\in f, q\in c}$ that is isomorphic to~$f\times c$. Indeed, 
$$\chi\big(\rho(p,q)\big) = \chi(\omega(q)p,q) = \psi_0\omega(q)\dotprod{\phi}{p}+\dotprod{\psi}{q} \le 1,$$
where equality holds if and only if~$\dotprod{\phi}{p}=1$ and~$\psi_0\omega(q)+\dotprod{\psi}{q}=1$, so that~$p\in f$ and~$q\in c$.

The above set~$\cF$ of facets of~$R$ contains all facets: indeed, any \face{(d+e-2)} of a facet in~$\cF$ is contained in precisely two facets in~$\cF$. Since the union of the edge sets of the facets in~$\cF$ is precisely~$G\times H$, it follows that the graph of~$R$ equals~$G\times H$.

A similar argument proves the same statement in the case when~$d>1$ and~$H$ is a regular subdivision of~$Q$ with internal vertices (meaning that not only the vertices of~$Q$ are lifted, but also a finite number of interior points).
\end{proof}

We already mentioned two examples obtained by such a construction in the beginning of this section (see \fref{fig:ctrm}): the product of a polytopal graph by a path and the product of a segment by a subdivision of an \gon{n} with no internal vertex. Proposition~\ref{prop:subdiv} even produces examples of regular polytopal products which are not simply polytopal:

\begin{example}\label{exm:diamond-subdiv}
Let~$H$ be the graph obtained by a star-clique operation from the graph of an octahedron. It is non-polytopal (Corollary~\ref{coro:starclique}), but it is the graph of a regular subdivision of a \poly{3}tope (see \fref{fig:truncatedoctahedron}). Consequently, the product of~$H$ by any regular polytopal graph is polytopal. Thus, there exist regular polytopal products which are not simply polytopal.

\begin{figure}[htb]
	\centerline{\includegraphics[width=\textwidth]{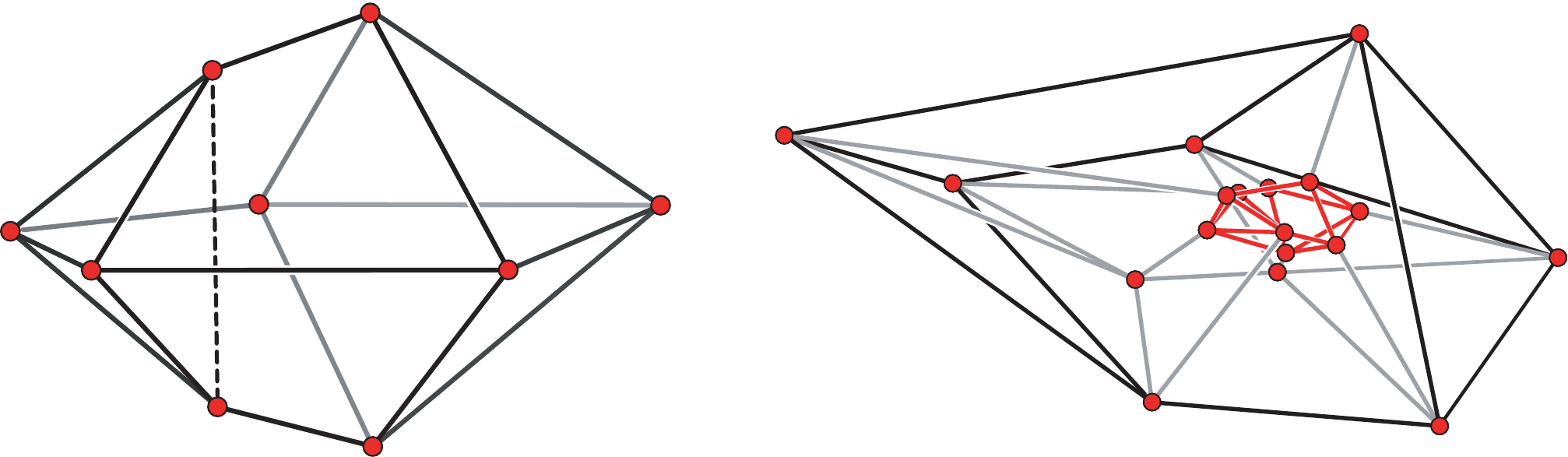}}
	\caption{A non-polytopal \regular{4} graph~$H$ which is the graph of a regular subdivision of a \poly{3}tope (left) and the Schlegel diagram of a \poly{4}tope whose graph is the product of~$H$ by a segment (right).}
	\label{fig:truncatedoctahedron}
\end{figure}

\end{example}

Finally, Proposition~\ref{prop:subdiv} also produces polytopal products of two non-polytopal graphs:

\begin{example}[Product of dominos]
Define the \defn{\domino{p} graph}~$D_p$ to be the product of a path~$P_p$ of length~$p$ by a segment. Let~$p,q\ge 2$. Observe that~$D_p$ and~$D_q$ are not polytopal and that~$D_p\times P_q$ is a regular subdivision of a \poly{3}tope. Consequently, the product of dominos ${D_p\times D_q}$ is a \poly{4}topal product of two non-polytopal graphs (see Figure~\ref{fig:domino}). 

Finally, let us observe that the product~$D_p\times D_q=P_p\times P_q\times (K_2)^2$ can be decomposed in different ways into a product of two graphs. However, in any such decomposition, at least one of the factors is non-polytopal.

\begin{figure}[htb]
	\centerline{\includegraphics[width=.39\textwidth]{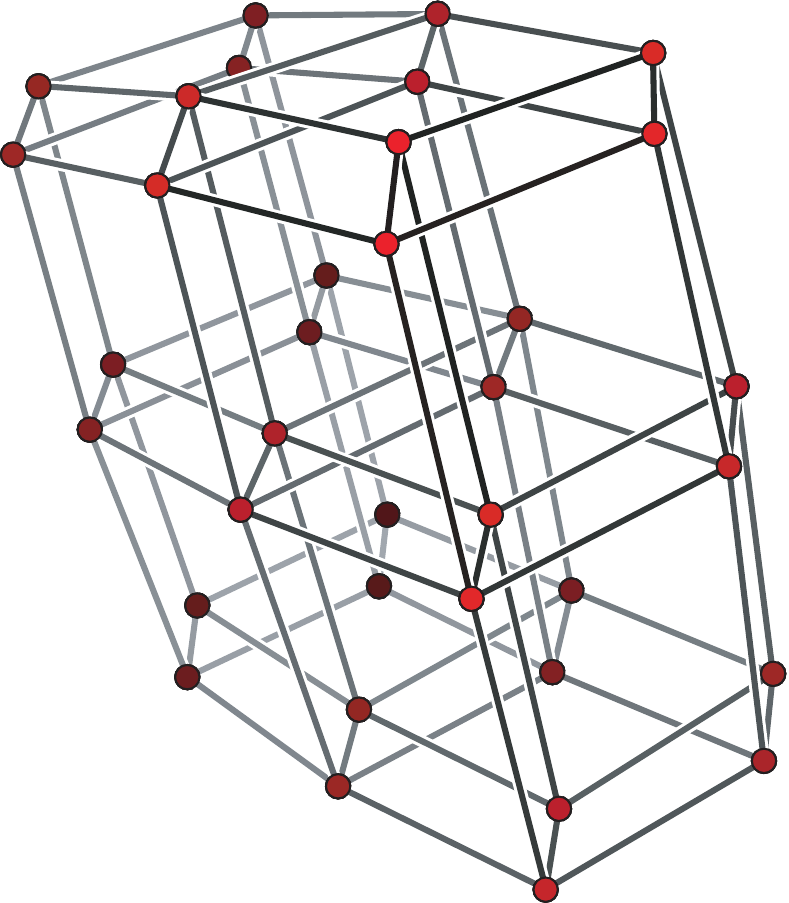} \qquad \includegraphics[width=.59\textwidth]{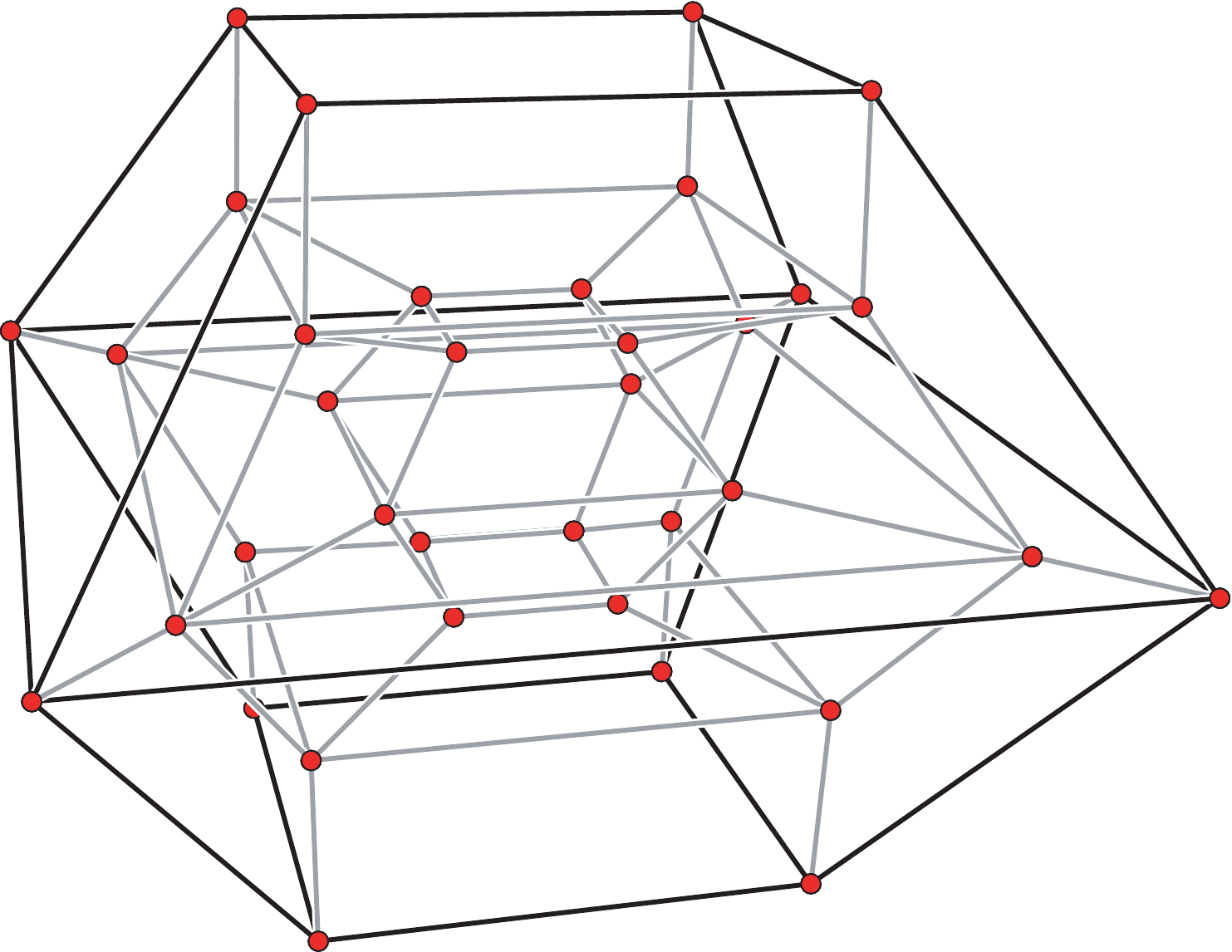}}
	\caption{The graph of the product of two \domino{2}s (left) and the Schlegel diagram of a realizing \poly{4}tope (right).}
	\label{fig:domino}
\end{figure}
\end{example}


\subsection{Product with a segment}\label{subsec:segment}

In this section, we complete our list of examples of products of a segment by a regular graph~$H$. The goal is to illustrate all possible behaviors of such a product regarding polytopality:
\begin{enumerate}
\item If~$H$ is polytopal, then~$K_2\times H$ is polytopal. However, some ambiguities can appear:
\begin{enumerate}
\item The dimension can be ambiguous. For example,~$K_2\times K_n$ is realized by the product of a segment by any neighborly polytope. See~\cite{mpp-psnp} for a discussion on dimensional ambiguity of products of complete graphs.
\item The dimension can be unambiguous, but the combinatorics of the polytope can be ambiguous. In this case, $H$~is not simply polytopal (Theorem~\ref{theo:kalai}). In Proposition~\ref{prop:prismoctahedron}, we determine all possible realizations of the graph of a prism over an octahedron.
\item There can be no ambiguity at all. This happens for example if~$H$ is simply \poly{3}topal.
\end{enumerate}
\item If~$H$ is not polytopal, then~$K_2\times H$ is not simply polytopal (Theorem~\ref{theo:simplypolytopalproducts}). However:
\begin{enumerate}
\item $K_2\times H$ can be polytopal in smaller dimension (Example~\ref{exm:diamond-subdiv}).
\item $K_2\times H$ can be non-polytopal. This happens for example when~$H$ is the complete graph~$K_{n,n}$ (Proposition~\ref{prop:knnxk2}) or when~$H$ is non-polytopal and \regular{3} (Proposition~\ref{prop:GxK2}).
\end{enumerate}
\end{enumerate}

\begin{proposition}\label{prop:knnxk2}
For~$n\ge3$, the graph~$K_2\times K_{n,n}$ is not polytopal.
\end{proposition}

To prove this proposition, we will need the following well known lemma:

\begin{lemma}\label{lem:8verts}
A \poly{3}tope with no triangular facet has at least~$8$~vertices.
\end{lemma}

\begin{proof}
Let~$P$ be a \poly{3}tope. For~$k\ge3$, denote by~$v_k$ the number of vertices of degree~$k$ and by~$p_k$ the number of \face{2}s with~$k$~vertices. By double counting and Euler's Formula (see \cite[Chapter~13]{g-cp-03} for details),
$$v_3+p_3 = 8+\sum_{k\ge5} (k-4)(v_k+p_k).$$
The lemma immediately follows.
\end{proof}

\begin{proof}[Proof of Proposition~\ref{prop:knnxk2}]
Observe that~$K_2\times K_{n,n}$ is not \poly{d}topal for~$d\le3$ because it contains a~$K_{3,3}$-minor, and for~$d=n+1$ by Theorem~\ref{theo:simplypolytopalproducts}.

The proof proceeds by contradiction. Suppose that~$K_2\times K_{n,n}$ is the graph of a \poly{d}tope~$P$, for some~$d$ with~$3\le d\le n$, and consider a \face{3}~$F$ of~$P$. Since ${K_2\times K_{n,n}}$~contains no triangle, Lemma~\ref{lem:8verts} says that $F$~has at least $8$~vertices. Denote by~$A$~and~$B$ the two maximal independent sets in~$K_{n,n}$, and by~$A_0,B_0,A_1,B_1$ their corresponding copies in the Cartesian product~$K_2\times K_{n,n}$. We discuss the possible repartition of the vertices of~$F$ in these sets.

Assume first that~$F$~has at least three vertices in~$A_0$; let~$x,y,z$ be three of them. Then it cannot have more than two vertices in~$B_0$, because otherwise its graph would contain a copy of~$K_{3,3}$. In fact, there must be exactly two vertices~$u,v$ in~$B_0$: since any vertex of~$F$ has degree at least~$3$, and each vertex in~$A_0$ can only be connected to vertices in~$B_0$ or to its corresponding neighbor in~$A_1$, each vertex of~$F$ in~$A_0$ must have at least, and thus exactly, two neighbors in~$B_0$ and one in~$A_1$. Thus,~$F$ also has at least three vertices in~$A_1$, and by the same reasoning, there must be exactly two vertices in~$B_1$; call one of them~$w$. But now~$\{x,y,z\}$ and~$\{u,v,w\}$ are the two maximal independent sets of a subdivision of~$K_{3,3}$ included in~$F$.

By symmetry and Lemma~\ref{lem:8verts}, $F$~has exactly two vertices in each of the sets $A_0,B_0,A_1,B_1$. Since all these vertices have degree~$3$, we have proved that~$P$'s only $3$-faces are combinatorial cubes whose graphs are Cartesian products of~$K_2$ with \cycle{4} in~$A_0\cup B_0=K_{n,n}$. However, this \cycle{4} is not contained in any other \face{3}, which is an obstruction to the existence of~$P$.
\end{proof}

\begin{proposition}\label{prop:GxK2}
If~$H$ is a non-polytopal and \regular{3} graph, then~$K_2\times H$ is non-polytopal.
\end{proposition}

\proof
We distinguish two cases:
\begin{enumerate}[(i)]
\item If~$H$ contains a~$K_4$-minor, then~$K_2\times H$ is not \poly{3}topal because it contains a~$K_5$-minor, and it is not \poly{4}topal by Theorem~\ref{theo:simplypolytopalproducts}. Since~$K_2\times H$ is \regular{4}, these are the only possibilities.
\item Otherwise,~$H$ is a series-parallel graph. Thus, it can be obtained from~$K_2$ by a sequence of \defn{series} and \defn{parallel} extensions, \ie subdividing or duplicating an edge. Since duplicating an edge creates a double edge, and subdividing an edge yields a vertex of degree two, $H$~is either not simple or not \regular{3}; since our graphs are simple by assumption, this case cannot occur.\qed
\end{enumerate}

To complete our collection of examples of products with a segment, we examine the possible realizations of the graph of the prism over the octahedron:

\begin{proposition}\label{prop:prismoctahedron}
The graph of the prism over the octahedron is realized by exactly four combinatorially different polytopes.
\end{proposition}

In order to exhibit four different realizations, we recall the situation and the proof of Proposition~\ref{prop:subdiv}. Given the graph~$G$ of a \poly{d}tope~$P$ and the graph~$H$ of a regular subdivision of an \poly{e}tope~$Q$ defined by a lifting function~$\omega:V(Q)\to\R$, we construct a \poly{(d+e)}tope with graph~$G\times H$ as follows: we start from the product~$P\times Q$ and we lift each face~$\{p\}\times Q$ using~$\omega$. This subdivides~$\{p\}\times Q$, creating the subgraph~$\{p\}\times H$ of the product~$G\times H$. Observe now that the deformation can be different at each vertex of~$P$: we can use a different lifting function at each vertex of~$P$, and produce combinatorially different polytopes.

To come back to our example, denote by~$H$ the graph of the octahedron. Observe that the octahedron has four regular subdivisions with no additional edges: the octahedron itself (for a constant lifting function), and the three subdivisions into two Egyptian pyramids glued along their square face (for a lifting function that vanishes in the common square face and is negative at the other two vertices). This leads to four combinatorially different realizations of~$K_2\times H$: in our previous construction, we can choose either the octahedron at both ends of the segment (thus obtaining the prism over the octahedron), or the octahedron at one end and the glued Egyptian pyramids at the other, or the glued Egyptian pyramids at both ends of the segment (and this leads to two possibilities according to whether we choose the same square or two orthogonal squares to subdivide the two octahedra).

\begin{figure}[htb]
	\centerline{\includegraphics[width=\textwidth]{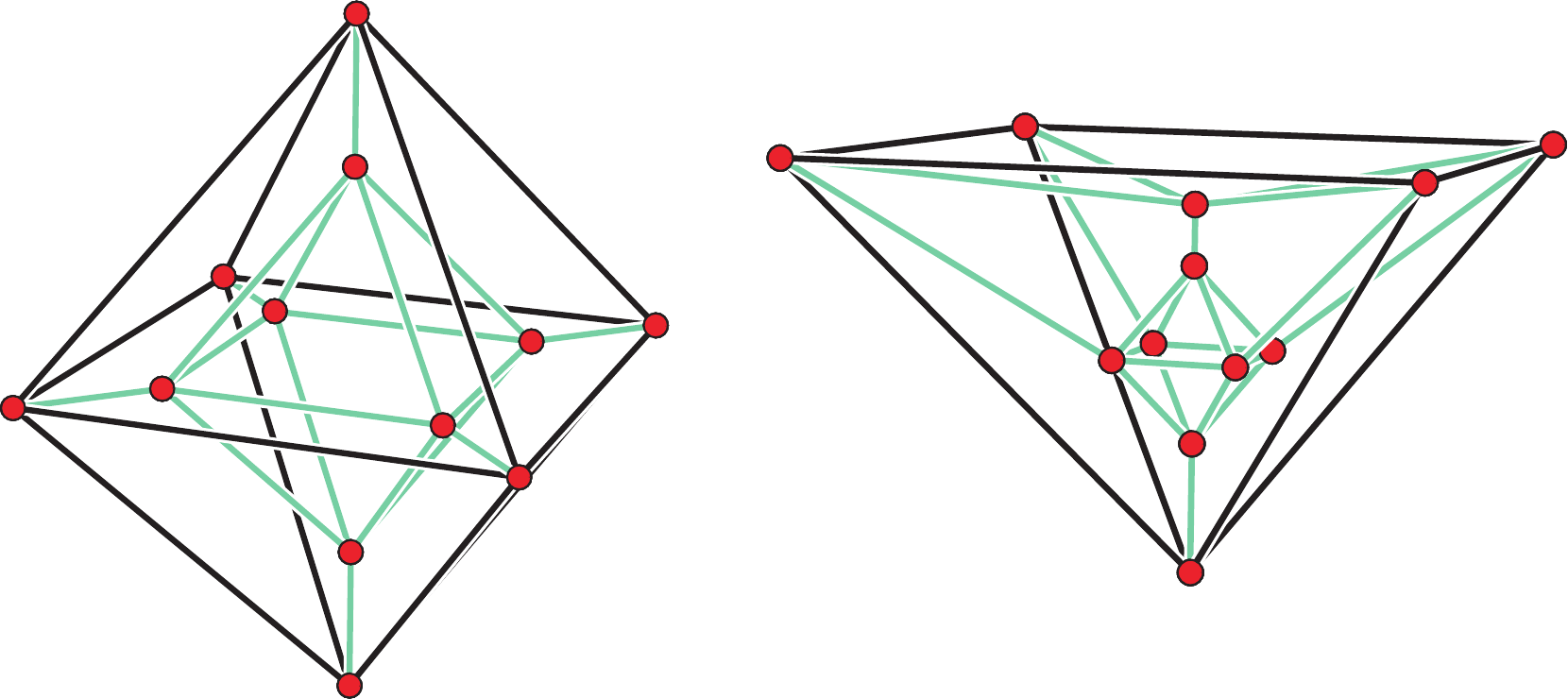}}
	\caption{The prism over the octahedron (left) and a combinatorially different polytope with the same graph (right).}
	\label{fig:prismoctahedron}
\end{figure}

In fact, by the same argument, we can even slightly improve Proposition~\ref{prop:subdiv}:

\begin{observation}
Let~$G$ be the graph of a \poly{d}tope~$P$ and~$H$ be the graph of an \poly{e}tope~$Q$. For each~$v\in G$, choose a lifting function~$\omega_v:V(Q)\to\R$, and denote by~$H_v$ the graph of the corresponding regular subdivision of~$Q$. Then the graph obtained by replacing in~$G\times H$ the subgraph~$\{v\}\times H$ by~$\{v\}\times H_v$ is polytopal.

The result remains true if we allow the use of extra, perhaps interior, points of $Q$ as vertices of the subdivisions, as long as~$d>1$ and all the subdivisions of~$Q$ have the same ones.
\end{observation}

It remains to prove that any realization of~$K_2\times H$ is combinatorially equivalent to one of the four described above.  First, the dimension is unambiguous:~$K_2\times H$ can only be \poly{4}topal (by Remark~\ref{remark:3poly} and Theorem~\ref{theo:simplypolytopalproducts}). In particular, any realization is almost simple, in the following sense:

\begin{definition}
A \poly{d}tope is \defn{almost simple} if its graph is \regular{(d+1)}.
\end{definition}

The vertex figures of a simple polytope are all simplices, which implies that any two incident edges in a simple polytope lie in a common \face{2}. For almost simple polytopes, the vertex figures are almost as restricted: they are \defn{$(d-1)$-circuits}, that is, \poly{(d-1)}topes with $d+1$ vertices. This implies the following property:

\begin{proposition}\label{prop:almostsimple}
Let~$\{v,w\}$ be an edge of an almost simple \poly{d}tope $P$. Then:
\begin{enumerate}[(a)]
\item either~$\{v,w\}$ together with any other edge incident to~$v$ forms a \face{2};
\item or there exists exactly one more edge incident to~$v$ which does not form a \face{2} with~$\{v,w\}$. In this case any two \face{2}s both incident to~$\{v,w\}$ lie in a \face{3}.
  \end{enumerate}
\end{proposition}

\begin{proof}
Consider the vertex figure~$F_v$ of~$v$. It is a \poly{(d-1)}tope with~$d+1$ vertices, one of which, say~$\overline{w}$, corresponds to the edge~$\{v,w\}$. This vertex~$\overline{w}$ can be adjacent to either~$d$ or~$d-1$ vertices of~$F_v$. The first case corresponds to statement~(a). In the second case, $\overline{w}$ has exactly one missing edge in~$F_v$ (corresponding to a missing \face{2} in~$P$), but the edge figure of~$\{v,w\}$ is a \simp{(d-2)}. This implies statement~(b).
\end{proof}

With this in mind, we can finally prove Proposition~\ref{prop:prismoctahedron}:

\begin{proof}[Proof of Proposition~\ref{prop:prismoctahedron}]
We introduce some notations: let~$V \eqdef \{1,\bar1,2,\bar2,3,\bar3\}$ denote the $6$ vertices of the octahedron such that~$\{1,\bar1\}$, $\{2,\bar2\}$ and $\{3,\bar3\}$ are the three missing edges, and let $a$ and $b$ denote the two endpoints of the segment factor. We denote the vertices of $K_2\times H$ by~$\{1a,1b,\bar1a,\dots,\bar3b\}$. We call \defn{horizontal edges} the edges of the form~$\{ia,ib\}$, for $i\in V$, and \defn{vertical edges} the edges of the form~$\{ix,jx\}$, for~$i\in V$, $j\in V\ssm\{i,\bar i\}$ and~$x\in\{a,b\}$.

We first use ``almost simplicity'' to study the possible \face{2}s of a realization~$P$ of~$K_2\times H$. Assume that there exists a \face{2} $F$ which is neither a triangle nor a square. It has to contain an angle between a horizontal edge and a vertical edge, say without loss of generality~$\{1a,1b\}$ and~$\{1a,2a\}$. By inducedness, the next edges of~$F$ are necessarily~$\{2a,\bar1a\}$ and~$\{\bar1a,\bar1b\}$. Since the edges~$\{1a,1b\}$ and~$\{1a,2a\}$ form an angle of~$F$, the two edges~$\{2a,2b\}$ and~$\{1b,2b\}$ cannot form an angle: otherwise the \cycle{4}~$(1a,1b,2b,2a)$ would form a square face which improperly intersects~$F$. Similarly, since the edges~$\{2a,\bar1a\}$ and~$\{\bar1a,\bar1b\}$ form an angle, the edges~$\{2a,2b\}$ and~$\{2b,\bar1b\}$ cannot form an angle. Thus,~$\{2a,2b\}$ is adjacent to two missing angles, which is impossible by Proposition~\ref{prop:almostsimple}. We conclude that the \face{2}s of any realization of~$K_2\times H$ can only be squares and triangles. 

We now use this information on the \face{2}s to understand the possible \face{3}s of~$P$. Assume that none of the angles of the \cycle{4}s~$(1a,2a,\bar1a,\bar2a)$, $(1a,3a,\bar1a,\bar3a)$, and~$(2a,3a,\bar2a,\bar3a)$ forms a \face{2}. Then for each~$i\in V$, the vertex~$ia$ has already two missing angles. Consequently, the remaining angles necessarily form a \face{2} of~$P$ by Proposition~\ref{prop:almostsimple}. By inducedness, we obtain all the triangles of the $a$-copy of $H$, and any two adjacent of these triangles are contained in a common \face{3}. This \face{3} is necessarily an octahedron.

Assume now that one of the angles of the \cycle{4}s~$(1a,2a,\bar1a,\bar2a)$, $(1a,3a,\bar1a,\bar3a)$, and $(2a,3a,\bar2a,\bar3a)$ forms a \face{2}. By symmetry, we can suppose that it is the angle defined by the edges~$\{1a,2a\}$ and~$\{2a,\bar1a\}$. Let~$F$ denote the corresponding \face{2} of~$P$. By inducedness, the last vertex of~$F$ cannot be either~$3a$ or~$\bar3a$, and~$F$ is necessarily the square~$(1a,2a,\bar1a,\bar2a)$. It is now easy to see that none of the angles of the \cycle{4}~$(1a,3a,\bar1a,\bar3a)$ (resp.~$(2a,3a,\bar2a,\bar3a)$) can be an angle of a \face{2} of $P$: otherwise, this \cycle{4} would be a \face{2} of~$P$ (by a symmetric argument), which would intersect improperly with~$F$. All together, this implies that the vertices~$3a$ and~$\bar3a$ both have already two missing angles, and thus, that all the other angles form \face{2}s by Proposition~\ref{prop:almostsimple}. Furthermore, any two \face{2}s adjacent to an edge~$\{3a,ia\}$, with~$i\in\{1,\bar1,2,\bar2\}$, form a \face{3}. This implies that all angles adjacent to a vertex~$ia$, except the angles of the \cycle{4}s~$(1a,3a,\bar1a,\bar3a)$ and~$(2a,3a,\bar2a,\bar3a)$ form a \face{2}.

Since the two above cases can occur independently at both ends of the segment~$(a,b)$, we obtain the claimed result.
\end{proof}


\subsection{Topological products}\label{subsec:topologicalproducts}

To finish, we come back to Ziegler's motivating question~\ref{qu:petersen}: ``is the product of two Petersen graphs polytopal?'' We proved in Theorem~\ref{theo:simplypolytopalproducts} that it is not~$6$ polytopal, but the question remains open in dimension~$4$ and~$5$.

\begin{proposition}\label{prop:pseudomanifold}
The product of two Petersen graphs is the graph of a cellular decomposition of $\RP^2 \times \RP^2$.
\end{proposition}

\begin{proof}
The Petersen graph is the graph of a cellular decomposition of the projective plane~$\RP^2$ with~$6$ pentagons (see \fref{fig:projectiveembeddings}). Consequently, the product of two Petersen graphs is the graph of a cellular decomposition of~$\RP^2 \times \RP^2$. The maximal cells of this decomposition are~$36$ products of two pentagons.
\end{proof}

\begin{figure}[h]
	\centerline{\includegraphics[scale=1]{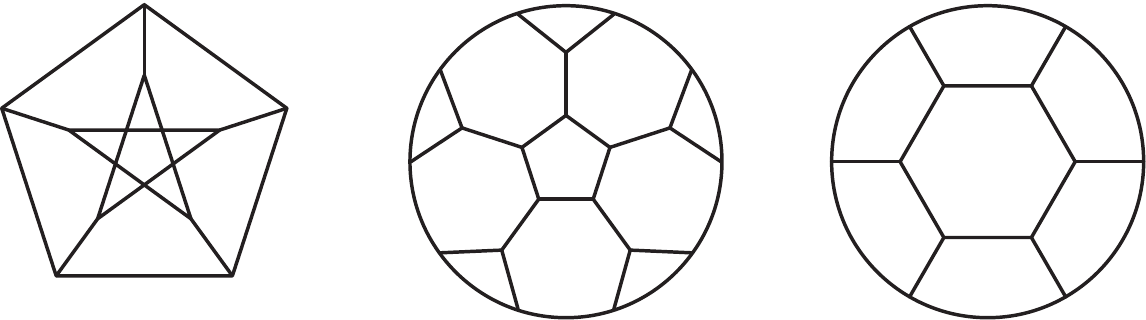}}
	\caption{The Petersen graph (left), its embedding on the projective plane (middle) and the embedding of~$K_{3,3}$ on the  projective plane (right). Antipodal points on the circle are identified.}
	\label{fig:projectiveembeddings}
\end{figure}

This proposition tells that understanding the possible \face{4}s of a realization, and their possible incidence relations (as we did for example in Proposition~\ref{prop:knnxk2}) is not enough to decide polytopality of the product of two Petersen graphs. Observe that the same remark holds for the product of any graphs of cellular decompositions of manifolds: for example, the product of a triangle by the Petersen graph is the graph of a cellular decomposition of~$\cS^1 \times \RP^2$.

Another interesting example is the product of a triangle by~$K_{3,3}$. Indeed, in contrast with $K_2 \times K_{3,3}$, the graph~$K_3 \times K_{3,3}$ is the graph of a cellular decomposition of the manifold~$\cS^1\times\RP^2$. To see this, embed~$K_{3,3}$ in the projective plane~$\RP^2$ as in \fref{fig:projectiveembeddings}, and multiply this embedding by a triangle. This cell decomposition is, however, not \emph{strongly regular}. Here, following~\cite{ekz-4p-03}, we say that a cell decomposition is strongly regular if every closed cell is embedded and the intersection of every two of them is a (perhaps empty) closed cell. Our decomposition fails to have the second property because the embedding of~$K_{3,3}$ in the projective plane already fails to have it: the central hexagon in the embedding of Figure~\ref{fig:projectiveembeddings} improperly intersects the three squares. Consequently, in the product with the triangle, each of the three hexagonal prisms improperly intersects three cubes. This can be solved by a ``Dehn surgery'', replacing the chain of three hexagonal prisms by a chain of six triangular prisms with the same boundary~---~see \fref{fig:decompose-S3}. In fact, it turns out that the cellular decomposition of $\cS^1\times\RP^2$ obtained in this way is the unique (strongly regular) combinatorial manifold whose graph is $K_{3,3} \times K_3$. This example and other results will be discussed in a future publication.

\begin{figure}[h]
	\centerline{\includegraphics[width=.8\textwidth]{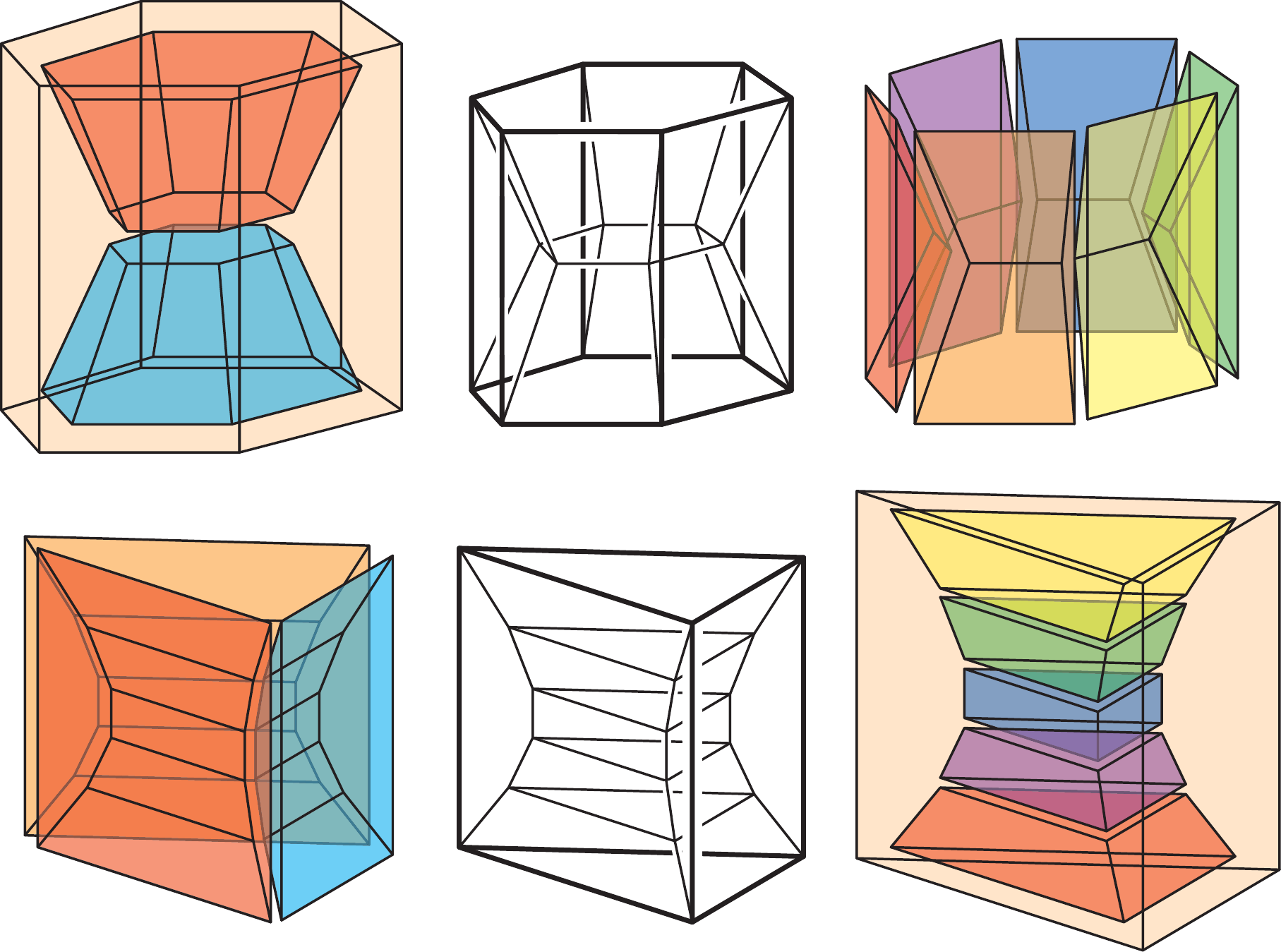}}
	\caption{A decomposition of the $3$-dimensional sphere (middle) using a cycle of three hexagonal prisms (left) and a cycle of six triangular prisms (right).}
	\label{fig:decompose-S3}
\end{figure}

\bibliographystyle{alpha}
\bibliography{biblio.bib}

\begin{thebibliography}{CRM09}

\bibitem[Bal61]{b-gscps-61}
Michel~L. Balinski.
\newblock On the graph structure of convex polyhedra in {$n$}-space.
\newblock {\em Pacific J.~Math.}, 11:431--434, 1961.

\bibitem[Bar67]{b-ncp-67}
David Barnette.
\newblock A necessary condition for {$d$}-polyhedrality.
\newblock {\em Pacific J.~Math.}, 23:435--440, 1967.

\bibitem[BML87]{bm-ppi-87}
Roswitha Blind and Peter Mani-Levitska.
\newblock Puzzles and polytope isomorphisms.
\newblock {\em Aequationes Math.}, 34(2-3):287--297, 1987.

\bibitem[CRM09]{crm}
CRM.
\newblock Report of the \emph{i-MATH Winter School DocCourse Combinatorics and
  Geometry 2009: Discrete and Computational Geometry}.
\newblock 2009.

\bibitem[CS99]{cs-ccpg-99}
Wen-Sz Chiue and Bih-Sheue Shieh.
\newblock On connectivity of the {C}artesian product of two graphs.
\newblock {\em Appl.~Math.~Comput.}, 102(2-3):129--137, 1999.

\bibitem[EKZ03]{ekz-4p-03}
David Eppstein, Greg Kuperberg, and G{\"u}nter~M. Ziegler.
\newblock Fat 4-polytopes and fatter 3-spheres.
\newblock In {\em Discrete Geometry: In honor of W.~Kuperberg's 60th birthday},
  volume 253 of {\em Monogr. Textbooks Pure Appl. Math.}, pages 239--265.
  Marcel Dekker, 2003.

\bibitem[Fri09]{f-fsppt-09}
Eric~J. Friedman.
\newblock Finding a simple polytope from its graph in polynomial time.
\newblock {\em Discrete Comput.~Geom.}, 41(2):249--256, 2009.

\bibitem[Gr{\"u}03]{g-cp-03}
Branko Gr{\"u}nbaum.
\newblock {\em Convex polytopes}, volume 221 of {\em Graduate Texts in
  Mathematics}.
\newblock Springer-Verlag, New York, second edition, 2003.
\newblock Prepared and with a preface by Volker Kaibel, Victor Klee and
  G{\"u}nter M.~Ziegler.

\bibitem[JZ00]{jz-ncp-00}
Michael Joswig and G{\"u}nter~M. Ziegler.
\newblock Neighborly cubical polytopes.
\newblock {\em Discrete Comput.~Geom.}, 24(2-3):325--344, 2000.
\newblock The Branko Gr{\"u}nbaum birthday issue.

\bibitem[Kal88]{k-swtsp-88}
Gil Kalai.
\newblock A simple way to tell a simple polytope from its graph.
\newblock {\em J.~Combin.~Theory Ser.~A}, 49(2):381--383, 1988.

\bibitem[Kle64]{k-ppg-64}
Victor Klee.
\newblock A property of {$d$}-polyhedral graphs.
\newblock {\em J.~Math.~Mech.}, 13:1039--1042, 1964.

\bibitem[MPP09]{mpp-psnp}
Benjamin Matschke, Julian Pfeifle, and Vincent Pilaud.
\newblock Prodsimplicial neighborly polytopes.
\newblock Accepted in \emph{Discrete Comput. Geom.} Available at
  \texttt{arXiv:0908.4177}, 2009.

\bibitem[NdO09]{gn-pc-09}
Marc Noy and Ant\'onio~Guedes de~Oliveira.
\newblock Personal communication.
\newblock 2009.

\bibitem[San10]{s-cehc-10}
Francisco Santos.
\newblock A counterexample to the {H}irsch conjecture.
\newblock Available at \texttt{arXiv:1006.2814}, 2010.

\bibitem[{\v{S}}pa08]{s-ccpg-08}
Simon {\v{S}}pacapan.
\newblock Connectivity of {C}artesian products of graphs.
\newblock {\em Appl.~Math.~Lett.}, 21(7):682--685, 2008.

\bibitem[Ste22]{s-pr-22}
Ernst Steinitz.
\newblock Polyeder und {R}aumeinteilungen.
\newblock In {\em Encyclop\"adie der mathematischen {W}issenschaften, {B}and 3
  ({G}eometrie), {T}eil {3AB12}}, pages 1--139. 1922.

\bibitem[SZ10]{sz-capdp}
Raman Sanyal and G{\"u}nter~M. Ziegler.
\newblock Construction and analysis of projected deformed products.
\newblock {\em Discrete Comput.~Geom.}, 43(2):412--435, 2010.

\bibitem[Whi32]{w-nspg-32}
Hassler Whitney.
\newblock Non-separable and planar graphs.
\newblock {\em Trans.~Amer.~Math.~Soc.}, 34(2):339--362, 1932.

\bibitem[Zie95]{z-lp-95}
G{\"u}nter~M. Ziegler.
\newblock {\em Lectures on polytopes}, volume 152 of {\em Graduate Texts in
  Mathematics}.
\newblock Springer-Verlag, New York, 1995.

\bibitem[Zie04]{z-ppp-04}
G{\"u}nter~M. Ziegler.
\newblock Projected products of polygons.
\newblock {\em Electron.~Res.~Announc. Amer.~Math.~Soc.}, 10:122--134
  (electronic), 2004.

\end{thebibliography}

\end{document}